\documentclass[reqno]{amsart}
\pdfoutput=1
\makeatletter
\let\origsection=\section \def\section{\@ifstar{\origsection*}{\mysection}} 
\def\mysection{\@startsection{section}{1}\z@{.7\linespacing\@plus\linespacing}{.5\linespacing}{\normalfont\scshape\centering\S}}
\makeatother        

\usepackage{amsmath,amssymb,amsthm}
\usepackage{mathrsfs}
\usepackage{dsfont}

\usepackage{todonotes}
\usepackage{verbatim}
 
\usetikzlibrary{arrows.meta}

\usepackage{graphicx}
\usepackage{xcolor}
\usepackage{tikz}
\usetikzlibrary{arrows}

\usepackage[backref]{hyperref}
\hypersetup{
    colorlinks,
    linkcolor={red!60!black},
    citecolor={green!60!black},
    urlcolor={blue!60!black},
    pdftitle={End-degree in digraphs},
    pdfauthor={Matthias Hamann, Karl Heuer}
}

\usepackage[open,openlevel=2,atend]{bookmark}

\usepackage[abbrev,msc-links,backrefs]{amsrefs} 
\usepackage{doi}

\renewcommand{\PrintDOI}[1]{\doi{#1}}

\usepackage[T1]{fontenc}
\usepackage{lmodern}
\usepackage[babel]{microtype}
\usepackage[english]{babel}

\usepackage{geometry}
\numberwithin{equation}{section}
\numberwithin{figure}{section}

\usepackage[shortlabels]{enumitem}

\newcommand{\cE}{\ensuremath{\mathcal{E}}}
\newcommand{\cP}{\ensuremath{\mathcal{P}}}
\newcommand{\cQ}{\ensuremath{\mathcal{Q}}}
\newcommand{\cR}{\ensuremath{\mathcal{R}}}
\newcommand{\cX}{\ensuremath{\mathcal{X}}}
\newcommand{\cY}{\ensuremath{\mathcal{Y}}}
\newcommand{\cZ}{\ensuremath{\mathcal{Z}}}

\let\polishlcross=\l
\def\l{\ifmmode\ell\else\polishlcross\fi}

\def\paragraph#1{%
  \noindent\textbf{#1.}\enspace}

\let\sm=\setminus

\newcommand{\N}{\ensuremath{\mathbb N}}

\makeatletter
\def\moverlay{\mathpalette\mov@rlay}
\def\mov@rlay#1#2{\leavevmode\vtop{   \baselineskip\z@skip \lineskiplimit-\maxdimen
   \ialign{\hfil$\m@th#1##$\hfil\cr#2\crcr}}}
\newcommand{\charfusion}[3][\mathord]{
    #1{\ifx#1\mathop\vphantom{#2}\fi
        \mathpalette\mov@rlay{#2\cr#3}
      }
    \ifx#1\mathop\expandafter\displaylimits\fi}
\makeatother

\DeclareFontFamily{U}  {MnSymbolC}{}
\DeclareSymbolFont{MnSyC}         {U}  {MnSymbolC}{m}{n}
\DeclareFontShape{U}{MnSymbolC}{m}{n}{
    <-6>  MnSymbolC5
   <6-7>  MnSymbolC6
   <7-8>  MnSymbolC7
   <8-9>  MnSymbolC8
   <9-10> MnSymbolC9
  <10-12> MnSymbolC10
  <12->   MnSymbolC12}{}
\DeclareMathSymbol{\powerset}{\mathord}{MnSyC}{180}

\let\epsilon=\varepsilon

\let\rho=\varrho
\let\theta=\vartheta
\let\phi=\varphi

\theoremstyle{plain}
\newtheorem{thm}{Theorem}[section]
\newtheorem{theorem}[thm]{Theorem}
\newtheorem{lemma}[thm]{Lemma}

\newtheorem{proposition}[thm]{Proposition}
\newtheorem{problem}[thm]{Problem}

\newtheorem*{claim*}{Claim}

\newtheorem{thm-intro}{Theorem}[]
\newtheorem{conj-intro}[thm-intro]{Conjecture}
\newtheorem{question-intro}[thm-intro]{Question}

\theoremstyle{definition}

\newtheorem{obs}[thm]{Observation}

\newtheorem*{example*}{Example}

\makeatletter
\newcommand\thankssymb[1]{\textsuperscript{\@fnsymbol{#1}}}
\makeatother

\begin{document}

\author[M.~Hamann]{Matthias Hamann\thankssymb{3}}
\address{Matthias Hamann, University of Hamburg, Department of Mathematics, Bundesstr. 55, 20146 Hamburg, Germany}
\email{\tt matthias.hamann@math.uni-hamburg.de}
\thanks{\thankssymb{3} Funded by the Deutsche Forschungsgemeinschaft (DFG) - Project No.\ 549406527.}

\author[K.~Heuer]{Karl Heuer\thankssymb{2}}
\address{Karl Heuer, Technical University of Denmark, Department of Applied Mathematics and Computer Science, Richard Petersens Plads, Building 322, 2800 Kongens Lyngby, Denmark}
\email{\tt karheu@dtu.dk}
\thanks{\thankssymb{2} Funded by Danmarks Frie Forskningsfond (DFF) - grant ID:\ 10.46540/5281-00242B}

\title[]{On ubiquity problems in infinite digraphs}



\begin{abstract}
We prove that the consistently oriented double ray is ubiquitous if and only if it is ubiquitous restricted to the class of one-ended digraphs.
Additionally, we prove the same equivalence for the disjoint union of a consistently forward and a consistently backward oriented ray. 
Furthermore, we discuss the connection between these two ubiquity problems.
\end{abstract}

\maketitle

\section{Introduction}
\label{sec:intro}

The property of a graph $H$ to be \emph{ubiquitous} describes the concept that the existence of $n$ pairwise disjoint copies for all $n\in\N$ of a graph $H$ in a graph $G$ implies that $G$ already contains $\aleph_0$ many pairwise disjoint copies of~$H$.
Originally, this was considered with respect to the subgraph relation and Halin showed that the ray is ubiquitous~\cite{Halin65}*{Satz 1} and that the double ray is ubiquitous~\cite{Halin70}*{Satz 3}.
Later, counterexamples \cites{A77, A02} regarding being ubiquitous were found, which even include trees.
Hence, ubiquity was also considered with respect to other relations such as the minor and topological minor relation, and we refer to \cites{A79,A80,BEEGHPTI,BEEGHPTII,BEEGHPTIII} for ubiquity results regarding these notions.
Many of these results leverage that finite graphs (resp.~trees) are well-quasi ordered under the minor (resp.~topological minor) relation.

It should be noted that Andreae~\cite{A02} posed the major and still open conjecture whether all locally finite connected graphs are ubiquitous under the minor relation. In the same paper he also provided an example of an uncountable non-ubiquitous graph under the minor relation, whose construction exploits that uncountable graphs are not well-quasi ordered under minors.

Further modified questions regarding ubiquity have been raised, e.g.~where disjointness is replaced by edge-disjointness.
In this setting it turned out to be complicated to even verify that double rays are ubiquitous regarding the subgraph relation~\cite{BCP15}.

For digraphs, much less is known, even for the subdigraph relation, and for the rest of this article, we shall only consider the property of being ubiquitous under that relation.
While Zuther~\cite{Z1997}*{Theorem 2.17} proved that the consistently oriented ray is ubiquitous, Gut et al.~\cite{GKR2024}*{Theorem 1.1} proved the same for all possible orientations of the undirected ray.
Later, Gut et al.~\cite{GKR2025}*{Theorem 1.3} classified almost all orientations of the undirected double ray that are ubiquitous.
The only orientation they left out was the consistently oriented one.
Due to that, they posed the following problem.

\begin{problem}\cite{GKR2024}*{Problem 1.3}\label{prob:ubiquity:doubleRays}
    Is the consistently oriented double ray ubiquitous?
\end{problem}

Our main result (Theorem~\ref{thm:main:intro:ReductionToOneEnd:doubleRays}) states that it suffices to verify Problem~\ref{prob:ubiquity:doubleRays} for one-ended digraphs.
(We refer to Section~\ref{sec:prelims} for the definition of ends of digraphs.)

\begin{theorem}\label{thm:main:intro:ReductionToOneEnd:doubleRays}
    The consistently oriented double ray is ubiquitous if and only if it is ubiquitous restricted to the class of one-ended digraphs.
\end{theorem}

As a major step in the proof of Theorem~\ref{thm:main:intro:ReductionToOneEnd:doubleRays}, we verify Problem~\ref{prob:ubiquity:doubleRays} for a large, but technical, class of digraphs, see Theorem~\ref{thm:ubiquity:doubleRays}.

Additionally, we consider the following more general version of \cite{HH2024+}*{Problem 4.2}.

\begin{problem}\label{prob:ubiquity:antiraysRays}
    Is the disjoint union of a consistently forward and a consistently backward oriented ray ubiquitous?
\end{problem}

We also verify Problem~\ref{prob:ubiquity:antiraysRays} for a large, but technical, class of digraphs (Theorem~\ref{thm:ubiquity:antiraysRays}) and show that Problem~\ref{prob:ubiquity:antiraysRays} holds if it holds for all one-ended digraphs.

\begin{theorem}\label{thm:main:intro:ReductionToOneEnd:RaysAntiRays}
    The disjoint union of a consistently forward and a consistently backward oriented ray is ubiquitous if and only if it is ubiquitous restricted to the class of one-ended digraphs.
\end{theorem}

\medskip

This paper is structured as follows. 
After introducing some terminology in Section~\ref{sec:prelims}, we prove Theorem~\ref{thm:main:intro:ReductionToOneEnd:RaysAntiRays} in Section~\ref{sec:ubiquity:antiraysRays}.
In Section~\ref{sec:ubiquity:DoubleRays}, we prove Theorem~\ref{thm:main:intro:ReductionToOneEnd:doubleRays}.
In Section~\ref{sec:equivalence}, we discuss connections between Problems~\ref{prob:ubiquity:doubleRays} and \ref{prob:ubiquity:antiraysRays}.

\section{Preliminaries}
\label{sec:prelims}

For general facts and notation regarding graphs we refer the reader to~\cite{diestel}, regarding digraphs in particular to~\cite{bang-jensen}.

For the sake of brevity we call a directed path just a \emph{dipath}.
Given a dipath $P$ containing two vertices $u, v$ such that $v$ is reached from $u$ via~$P$, we define $uPv$ as the subdipath of~$P$ starting at~$u$ and ending in~$v$.
Given two vertex sets $A$, $B$, we call a dipath $P$ an \emph{$A$--$B$ dipath} if $P$ starts at a vertex of~$A$, ends at a vertex of~$B$ and is internally disjoint from $A \cup B$.

A digraph $R$ is called a \emph{ray} if it has precisely one vertex $v$ with out-degree $1$ and in-degree~$0$, there exists a dipath from $v$ to each vertex of~$R$, and each vertex distinct from~$v$ has
out-degree $1$ and in-degree~$1$.
An \emph{anti-ray} is obtained from a ray by reversing all its edges.
The vertex $v$ is called the \emph{starting vertex} (resp.\ \emph{end vertex}) of the ray (resp.\ anti-ray).
A \emph{tail} of~$R$ is a subray of~$R$.
If this tail starts at the vertex~$x$, then we denote it by~$xR$.

Let $Q$ and~$R$ be rays or anti-rays.
We write $Q\leq R$ if there are infinitely many pairwise disjoint $Q$--$R$ dipaths and we write $Q\sim R$ if $Q\leq R$ and $R\leq Q$.
Then $\leq$ is a preorder on the set of rays and anti-rays in a digraph~$D$ and $\sim$ is an equivalence relation on that set.
The equivalence classes of~$\sim$ are the \emph{ends} of~$D$ and we can extend the relation $\leq$ to the ends: we write $\epsilon\leq\eta$ for ends $\epsilon$ and $\eta$ if there are $Q\in\epsilon$ and $R\in\eta$ with $Q\leq R$.
Note that $\epsilon\leq\eta$ if and only if $Q\leq R$ for every $Q\in\epsilon$ and $R\in\eta$.
In particular, we have $\epsilon\leq\eta$ and $\eta\leq\epsilon$ if and only if $\epsilon=\eta$.
Note that $\leq$ is an order on the ends of~$D$.

A \emph{double ray} is an infinite digraph all of whose vertices have in- and out-degree~$1$ and such that for every $v,w$ there exists either a $v$--$w$ or a $w$--$v$ dipath.
For every vertex $x$ on a double ray~$R$, the subdigraphs that are rays starting at~$x$ or anti-rays ending in~$x$ are the \emph{tails} of~$R$, more specifically, the \emph{forward} and \emph{backward tails} of~$R$, and we denote them by $xR$ and $Rx$, respectively.
If $\epsilon$ and~$\eta$ are ends and $R$ is a double ray such that some tail of~$R$ that is a ray lies in~$\eta$ and some tail of~$R$ that is an anti-ray lies in~$\epsilon$ then we call $R$ an \emph{$\epsilon$--$\eta$ double ray}.

Let $X$ and~$Y$ be disjoint sets of ends of~$D$.
We say that a vertex set $S \subseteq V(D)$ \emph{separates} $X$ from~$Y$ in~$D$ if for every $\epsilon\in Y$ every $R\in\epsilon$ has a tail~$Q$ such that every ray in elements of~$X$ that starts at a vertex of~$Q$ meets~$S$.
In case $X$ (or $Y$) is a singleton set, we ease the notation and analogously define that the end $\epsilon_X \in X$ (or $X$) is separated from $Y$ (or from the end $\epsilon_Y \in Y$).
Note that if $\epsilon\not\leq\eta$, then there exists a finite vertex set separating $\eta$ from~$\epsilon$.

We call the digraph in Figure~\ref{fig:bidirectedQuarterGrid} the \emph{bidirected quarter-grid}.
Formally, this digraph is defined as follows.
The first digraph is built from infinitely many pairwise disjoint rays $R_1=x_1^1x_2^1\ldots$, $R_2=x_1^2x_2^2\ldots$, $\ldots$ where we add edges $x_{4j+7}^ix_{4j+1}^{i+1}$ and $x_{4j+2}^{i+1}x_{4j+8}^i$ for all $j\geq 0$ and $i\geq 1$.
Finally, we suppress all vertices $v\in V(R_i)$ with $d^-(v) = d^+(v) = 1$ that have only neighbours on~$R_i$ to obtain the bidirected quarter-grid.
Furthermore, we call the digraph obtained from the bidirected quarter-grid by reversing all orientations of the edges the \emph{reversed bidirected quarter-grid}.

\begin{figure}[ht]
    \centering
\begin{tikzpicture}
    \draw[-{Latex[length=2.75mm]}] (1.5,0) -- (8.25,0);
    \foreach \x in {1.5,2,2.5,3,3.5,5,5.5,7,7.5}
        \fill (\x,0) circle (2pt);

    \draw (1,-0.3) node[anchor=south] {$R_1$};

    \draw[-{Latex[length=2.75mm]}] (2,1) -- (9.25,1);
    \foreach \x in {2,2.5,3,3.5,4,4.5,5,5.5,6,6.5,7,7.5,8,8.5}
        \fill (\x,1) circle (2pt);

    \draw (1.5,0.7) node[anchor=south] {$R_2$};

    \draw[-{Latex[length=2.75mm]}] (4,2) -- (10.25,2);
    \foreach \x in {4,4.5,6,6.5,7,7.5,8,8.5,9,9.5}
        \fill (\x,2) circle (2pt);

    \draw (3.5,1.7) node[anchor=south] {$R_3$};

    \draw[-{Latex[length=2.75mm]}] (7,3) -- (11.25,3);
    \foreach \x in {7,7.5,9,9.5,10,10.5}
        \fill (\x,3) circle (2pt);

    \draw (6.5,2.7) node[anchor=south] {$R_4$};


    \draw[-{Latex[length=2.75mm]}] (10,3) -- (10,3.75);
    \draw[-{Latex[length=2.75mm]}] (10.5,3.75) -- (10.5,3);

    \draw[-{Latex[length=2.75mm]}] (2,0) -- (2,1);
    \draw[-{Latex[length=2.75mm]}] (3,0) -- (3,1);
    \draw[-{Latex[length=2.75mm]}] (5,0) -- (5,1);
    \draw[-{Latex[length=2.75mm]}] (7,0) -- (7,1);    
    
    \draw[-{Latex[length=2.75mm]}] (2.5,1) -- (2.5,0);
    \draw[-{Latex[length=2.75mm]}] (3.5,1) -- (3.5,0);
    \draw[-{Latex[length=2.75mm]}] (5.5,1) -- (5.5,0);
    \draw[-{Latex[length=2.75mm]}] (7.5,1) -- (7.5,0);    
        

    \draw[-{Latex[length=2.75mm]}] (4,1) -- (4,2);
    \draw[-{Latex[length=2.75mm]}] (6,1) -- (6,2);

    \draw[-{Latex[length=2.75mm]}] (4.5,2) -- (4.5,1);

    \draw[-{Latex[length=2.75mm]}] (6.5,2) -- (6.5,1);

    \draw[-{Latex[length=2.75mm]}] (8,1) -- (8,2);
    \draw[-{Latex[length=2.75mm]}] (8.5,2) -- (8.5,1);    

    \draw[-{Latex[length=2.75mm]}] (7,2) -- (7,3);  
    \draw[-{Latex[length=2.75mm]}] (7.5,3) -- (7.5,2);    
    
    \draw[-{Latex[length=2.75mm]}] (9,2) -- (9,3);  
    \draw[-{Latex[length=2.75mm]}] (9.5,3) -- (9.5,2);

    \draw (8.25,0.5);
    \fill (8.25,0.5) circle (1pt);
    \draw (8.5,0.5);
    \fill (8.5,0.5) circle (1pt);
    \draw (8.75,0.5);
    \fill (8.75,0.5) circle (1pt);

    \draw (9.25,1.5);
    \fill (9.25,1.5) circle (1pt);
    \draw (9.5,1.5);
    \fill (9.5,1.5) circle (1pt);
    \draw (9.75,1.5);
    \fill (9.75,1.5) circle (1pt);

    \draw (10.25,2.5);
    \fill (10.25,2.5) circle (1pt);
    \draw (10.5,2.5);
    \fill (10.5,2.5) circle (1pt);
    \draw (10.75,2.5);
    \fill (10.75,2.5) circle (1pt);

    \draw (11.25,3.5);
    \fill (11.25,3.5) circle (1pt);
    \draw (11.5,3.65);
    \fill (11.5,3.65) circle (1pt);
    \draw (11.75,3.8);
    \fill (11.75,3.8) circle (1pt);

\end{tikzpicture}
    \caption{The bidirected quarter-grid.}
    \label{fig:bidirectedQuarterGrid}
\end{figure}
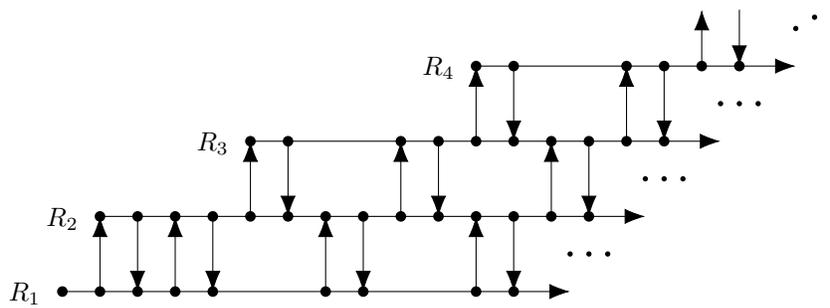

The following lemma follows directly from the structure of the (reversed) bidirected quarter-grid.

\begin{obs}\label{obs:raysInBiDirQGrid}
    Let $D$ be the subdivision of a (reversed) bidirected quarter-grid and let $U$ be an infinite vertex set in~$D$.
    Then there exist infinitely many pairwise disjoint (anti-)rays in~$D$ such that their end vertices lie in~$U$.\qed
\end{obs}

A sufficient condition for the existence of (reversed) quarter-grids was given in~\cite{HH2024+b}.

\begin{thm}\label{thm:grids}\cite{HH2024+b}*{Theorem 1.2}
    If $D$ is a digraph that contains an end~$\epsilon$ with infinitely many disjoint (anti-)rays, then there exists a subdivision of the (reversed) bidirected quarter-grid in~$D$ with all its (anti-)rays in~$\epsilon$.
\end{thm}

We will need the following results later.

\begin{lemma}\label{prop:rays+anti=double}\cite{HH2024+}*{Proposition 4.4}
For every $n\in\mathbb N$, if a digraph $D$ contains a set of $n$ rays and $n$ anti-rays that are all pairwise disjoint and lie in the same end, then there exists a set of $n$ pairwise disjoint double rays in~$D$ all of whose tails lie in that end.
\end{lemma}

\begin{thm}\label{thm:endDegree}\cite{HH2024+}*{Theorem 3.1}
    Let $D$ be a digraph.
    If an end of~$D$ contains n pairwise disjoint rays (anti-rays) for all $n\in\N$, then it contains countably infinitely many pairwise disjoint rays (anti-rays).
\end{thm}

The last result that we need is by Zuther \cite{Z1998}.

\begin{thm}\label{thm:Zuther:grids}\cite{Z1998}*{Theorem 3.1}
    Let $D$ be a digraph and let $(\epsilon_i)_{i\in\N}$ be a strictly increasing sequence of ends in~$D$ each of which contains rays (anti-rays).
    Then there exists the supremum $\epsilon$ of this sequence and it contains infinitely many pairwise disjoint rays (anti-rays).
\end{thm}

\section{On the ubiquity for a pair of an anti-ray and a ray}\label{sec:ubiquity:antiraysRays}

In this section, we will verify a special case of Problem~\ref{prob:ubiquity:antiraysRays} and obtain Theorem~\ref{thm:main:intro:ReductionToOneEnd:RaysAntiRays} as a corollary.
For this, we will need the following definition.
We say that a sequence $(\cX_n)_{n\in\N}$, where, for every $n\in\N$, $\cX_n$ is a set of~$n$ rays and $n$ anti-rays in a digraph $D$ that are all pairwise disjoint, \emph{concentrates} in an end $\epsilon$ of~$D$ if there exists $N\in\N$ such that there are at most $N$ rays and at most $N$ anti-rays in each~$\cX_n$ that do not lie in~$\epsilon$.

\begin{thm}\label{thm:ubiquity:antiraysRays}
    Let $D$ be a digraph.
    If there is a sequence $(\cX_n)_{n\in\N}$ such that, for every $n\in\N$, the set $\cX_n$ consists of $n$ rays and $n$ anti-rays of~$D$ that are all pairwise disjoint and such that the sequence $(\cX_n)_{n\in\N}$ does not concentrate in any end of~$D$, then $D$ contains a set of infinitely many rays and infinitely many anti-rays every two of which are disjoint.
\end{thm}

\begin{proof}
    For every $n\in\N$, let $X_n$ be a set of $n$ rays and let $Y_n$ be a set of $n$ anti-rays such that all elements of $\cX_n:=X_n\cup Y_n$ are pairwise disjoint and such that $(\cX_n)_{n\in\N}$ does not concentrate in any end of~$D$.
    Let us first assume that there exists infinitely many rays $R_1,R_2,\ldots\in\bigcup_{n\in\N} X_n$ all in distinct ends.
    If there are also infinitely many anti-rays $Q_1,Q_2,\ldots\in\bigcup_{n\in\N} Y_{n}$ all in distinct ends, then we may pass to an infinite subset of both of them such that no ray $R_i$ is equivalent to any anti-ray~$Q_j$.
    Then, we can successively take tails of each $R_i$ and $Q_i$ for all $i\in\N$ such that they are disjoint from all previously chosen tails.
    Thus, we obtain a set consisting of infinitely many rays and infinitely many anti-rays such that all of them are pairwise disjoint.

    If there are not infinitely many anti-rays in $\bigcup_{n\in\N} Y_{n}$ all from distinct ends, then there are finitely many ends that contain all anti-rays in $\bigcup_{n\in\N} Y_{n}$.
    In particular, there exists an end $\epsilon$ that contains $m$ pairwise disjoint anti-rays from $\bigcup_{n\in\N} Y_{n}$ for all $m\in\N$.
    By Theorem~\ref{thm:endDegree}, this end contains infinitely many anti-rays $Q_1, Q_2, \ldots$ that are all pairwise disjoint.
    We may assume that no $R_i$ lies in~$\epsilon$.
    Then, we can again take tails of each $R_i$ and $Q_i$ for all $i\in\N$ successively and obtain infinitely many rays and infinitely many anti-rays such that all of them are pairwise disjoint.
    This finishes the case that there exists infinitely many rays from $\bigcup_{n\in\N} X_n$ all in distinct ends.
    
    By a symmetric argument, we also obtain infinitely many rays and infinitely many anti-rays that are all pairwise disjoint in the case that that there exist infinitely many anti-rays from $\bigcup_{n\in\N} Y_n$ all in distinct ends.

    Thus, all elements of $\bigcup_{n\in\N} (X_n\cup Y_n)$ lie in only finitely many ends.
    By the assumption on $(\cX_n)_{n\in\N}$, there exists an end $\epsilon$ and a different end~$\eta$ such that $\bigcup_{n\in\N} X_n$ contains $k$ pairwise disjoint rays for every $k\in\N$ from~$\epsilon$ and $\bigcup_{n\in\N} Y_n$ contains $k$ pairwise disjoint anti-rays for every $k\in\N$ from~$\eta$.
    By Theorem~\ref{thm:endDegree}, there are infinitely many pairwise disjoint anti-rays $Q_1,Q_2,\ldots$ in~$\eta$ and infinitely many pairwise disjoint rays $R_1,R_2,\ldots$ in~$\epsilon$.
    By successively taking tails of $R_i$ and $Q_i$ for all $i\in\N$ such that they are disjoint from all previously chosen tails, we also obtain in this case a set of rays and anti-rays as desired.
    Thus, we have found such a set in each case, which shows the assertion.
\end{proof}

The property for sequences of concentrating in an end seems to imply that there is a unique end where it suffices to look at.
This is indeed the case and, moreover, this allows us to construct a one-ended digraph in which we find a sequence of sets of rays and anti-rays as in the premise of the ubiquity condition.
Let us make this precise in our next lemma.

\begin{lemma}\label{lem:reduction:oneEnd}
    Let $D$ be a digraph and let $(\cX_n)_{n\in\N}$ be a sequence such that, for each $n\in\N$, the set $\cX_n$ consists of~$n$ rays and $n$ anti-rays in~$D$ that are all pairwise disjoint.
    If there is an end $\epsilon$ such that all $R\in\cX_n$ for all $n\in\N$ lie in~$\epsilon$, then there exists a one-ended subdigraph $D'$ of~$D$ such that there is for every $n\in\N$ a set $\cY_n$ consisting of $n$ rays and $n$ anti-rays in~$D'$ that are all pairwise disjoint.
\end{lemma}

\begin{proof}
    Let $R_1,R_2,\ldots$ be an enumeration of the rays in $\bigcup_{n\in\N}\cX_n$ and let $Q_1,Q_2,\ldots$ be an enumeration of the anti-rays in $\bigcup_{n\in\N}\cX_n$ such that $i<j$ implies $n_i\leq n_j$ for the smallest $n_i$, $n_j$ such that $\cX_{n_i}$, $\cX_{n_j}$ contains $R_i$, $R_j$ (or $Q_i$, $Q_j$), respectively.
    By shortening each $R_i$ or~$Q_i$, if necessary, we may assume that no vertex lies on infinitely many $R_i$ or~$Q_i$.
    We join each pair $R_i$, $R_{i+1}$ by infinitely many pairwise disjoint dipaths in both directions, i.\,e.\ from $R_i$ to $R_{i+1}$ and from $R_{i+1}$ to~$R_i$, such that these dipaths are arranged as the vertical edges of the bidirected quarter-grid in Figure~\ref{fig:bidirectedQuarterGrid}.
    For the anti-rays $Q_i$ and $Q_{i+1}$ we do the same with respect to the reversed bidirected quarter-grid.
    Lastly, we join $R_1$ and~$Q_1$ by infinitely many pairwise disjoint dipaths in both directions.
    By choosing all these dipaths one after the other along an enumeration, we can achieve that each of these dipaths is disjoint from all previously chosen ones.
    Hence, all these dipaths are pairwise disjoint.
    We call the resulting subdigraph~$D'$.
    Obviously, $D'$ admits the sets $\cY_n$ as claimed in that we may take $\cY_n$ such that each element of it is a tail of some element of~$\cX_n$.
    Thus, it remains to prove that $D'$ has a unique end.
    Note that all $R_i$ and $Q_j$ are equivalent by construction.
    Now, let $R$ be ray or anti-ray in~$D'$.
    By construction, $R$ meets either one $R_i$ or one $Q_i$ infinitely often or infinitely many $R_i$ or infinitely many $Q_i$.
    In the first case, $R$ and $R_i$ or $R$ and $Q_i$ are obviously equivalent.
    In the second case, it is easy to construct infinitely many pairwise disjoint $R_1$--$R$ and infinitely many pairwise disjoint $R$--$R_1$ dipaths, where we use that each vertex lies on only finitely many of the~$R_i$.
    The third case is symmetric to the second one.
    Thus, all rays and anti-rays in~$D'$ are equivalent.   
\end{proof}

As a corollary of Lemma~\ref{lem:reduction:oneEnd} and Theorem~\ref{thm:ubiquity:antiraysRays}, we obtain Theorem~\ref{thm:main:intro:ReductionToOneEnd:RaysAntiRays}.

\begin{proof}[Proof of Theorem~\ref{thm:main:intro:ReductionToOneEnd:RaysAntiRays}.]
    Let us assume that Problem~\ref{prob:ubiquity:antiraysRays} holds for the class of one-ended digraphs and let $D$ be a digraph that admits a sequence $(\cX_n)_{n\in\N}$ such that, for every $n\in\N$, the set $\cX_n$ consists of $n$ rays and $n$ anti-rays that are pairwise disjoint.
    We may assume that $(\cX_n)_{n\in\N}$ concentrates in some end $\epsilon$ as otherwise the assertion follows from Theorem~\ref{thm:ubiquity:antiraysRays}.
    Then we find a sequence $(\cY_n)_{n\in\N}$ such that $\cY_n$, for every $n\in\N$, is a set of $n$ rays and $n$ anti-rays that are all pairwise disjoint and all lie in~$\epsilon$.
    Then there exists a one-ended subdigraph $D'$ of~$D$ by Lemma~\ref{lem:reduction:oneEnd} that also admits a sequence $(\cZ_n)_{n\in\N}$ such that $\cZ_n$, for every $n\in\N$, is a set of $n$ rays and $n$ anti-rays in~$D'$ that are all pairwise disjoint.
    By assumption, there exists a set consisting of infinitely many rays and infinitely many anti-rays in~$D'$ that are all pairwise disjoint.
    In particular, such a set exists in~$D$.
\end{proof}

\section{On the ubiquity for double rays}\label{sec:ubiquity:DoubleRays}

In this section, we will prove our main result, Theorem~\ref{thm:main:intro:ReductionToOneEnd:doubleRays}.
For this, we consider the situation that a digraph contains $n$ pairwise disjoint double rays for all $n\in\N$.
For this, we call a sequence $(\mathcal R_n)_{n\in\N_{\geq N}}$ with $N\in\N$ such that $\mathcal R_n$ is a set of $n$ pairwise disjoint double rays in~$D$ a \emph{valid} sequence.
It is often useful to pass onto new sequences that arise from $(\mathcal R_n)_{n\in\N}$ in a certain way.
For this, a sequence $(\mathcal P_i)_{i\in\N}$ is a valid sequence \emph{induce by} $(\mathcal R_n)_{n\in\N}$ if there exists a strictly increasing sequence $(n_i)_{i\in\N}$ such that $\mathcal P_i$ is a subset of size~$i$ of~$\mathcal R_{n_i}$.

A valid sequence $(\cR_n)_{n\in\N}$ \emph{concentrates} in finitely many ends if there exist $N,M\in\N$ and $M$ ends $\epsilon_1,\ldots,\epsilon_M$ of~$D$ such that for all $n\in\N$ for every double ray $R\in\cR_n$ except for at most~$N$ there exists $1\leq i\leq M$ such that all tails of~$R$ lie in~$\epsilon_i$.
Furthermore, we say that the rays (the anti-rays) in $(\cR_n)_{n\in\N}$ \emph{concentrate} in a set $\mathcal E$ of ends if there exists $N\in\N$ such that every forward tail (every backward tail) from each $\cR_n$ with $n\geq N$ lies in an element of~$\mathcal E$.
If $\epsilon$ is the only element of~$\cE$, we simply write that the rays (the anti-rays) concentrate in~$\epsilon$.

It is well-known that every infinite poset contains either a countably infinite anti-chain, or an infinite strictly increasing chain isomorphic to $\N$, or an infinite strictly decreasing chain isomorphic to $\mathbb Z\sm\N$.
We need a similar but stronger result, see Proposition~\ref{prop:ConcentrationSomewhere}.
Before we state that proposition, we need some more notations for its proof.

Let $\epsilon$ be an end of a digraph~$D$.
Then we define the following sets
\begin{alignat*}{2}
    &\lfloor\epsilon\rfloor^\ast&&:=\{\eta>\epsilon\mid \eta\text{ end of }D\},\\
    &\lceil\epsilon\rceil^\ast&&:=\{\eta<\epsilon\mid \eta\text{ end of }D\},\\
    &\epsilon^\bot&&:=\{\eta\mid \eta\text{ end of }D:\ \eta\not\leq\epsilon\text{ and }\eta\not\geq\epsilon\}.
\end{alignat*}

\begin{proposition}\label{prop:ConcentrationSomewhere}
    Let $D$ be a digraph and let $(\cR_n)_{n\in\N}$ be a valid sequence.
    Then there is a valid sequence $(\cQ_n)_{n\in\N}$ induced by $(\cR_n)_{n\in\N}$ such that for the set $\cE_n$ of ends containing the forward tails (the backward tails) of the double rays in~$\cQ_n$ one of the following holds:
    \begin{enumerate}[label = {\rm (\roman*)}]
        \item\label{itm:ConcentrationSomewhere:1} there is a unique end $\epsilon$ with $\{\epsilon\}=\cE_n$ for all $n\in\N$;
        \item\label{itm:ConcentrationSomewhere:2} $\epsilon_n<\epsilon_m$ for all $n<m\in\N$ and $\epsilon_n\in\cE_n$ and $\epsilon_m\in\cE_m$;
        \item\label{itm:ConcentrationSomewhere:3} $\epsilon_n>\epsilon_m$ for all $n<m\in\N$ and $\epsilon_n\in\cE_n$ and $\epsilon_m\in\cE_m$;
        \item\label{itm:ConcentrationSomewhere:4} $\epsilon_n$ and $\epsilon_m$ are incomparable for all $n\neq m\in\N$ and $\epsilon_n\in\cE_n$ and $\epsilon_m\in\cE_m$.
    \end{enumerate}
\end{proposition}

\begin{proof}
    If suffices to prove the proposition for the forward tails: the other case follows from this by applying the result for the forward tails to the digraph obtained from~$D$ by reversing all edge directions.

    Let us assume that no valid sequence induced by $(\cR_n)_{n\in\N}$ satisfies~\ref{itm:ConcentrationSomewhere:1}.
    Then, in particular, for no valid sequence induced by $(\cR_n)_{n\in\N}$, the rays concentrate in any end.
    Let $R_1^1\in\cR_1$, set $\cQ'_1:=\{R_1^1\}$, and let $\eta_1^1$ be the end that contains the forward tail of~$R_1^1$.
    Then there is a valid sequence $(\cR_n^1)_{n\in\N_{\geq 2}}$ induced by $(\cR_n)_{n\in\N_{\geq 2}}$ all of whose forward tails either lie in $\lfloor\eta_1^1\rfloor^\ast$, or lie in $\lceil\eta_1^1\rceil^\ast$, or lie in $(\eta_1^1)^\bot$.
    
    If there exists an $n\in\N$ such that $\cR^1_n$ has two double rays $R_1^2$ and $R_2^2$ whose forward tails lie in the same end, then we denote this end by~$\eta_1^2$.
    As before, we find a valid sequence $(\cR^2_n)_{n\in\N_{\geq 3}}$ induced by $(\cR^1_n)_{n\in\N_{\geq 3}}$ all of whose forward tails either lie in $\lfloor\eta_1^2\rfloor^\ast$, or lie in $\lceil\eta_1^2\rceil^\ast$, or lie in $(\eta_1^2)^\bot$.
    We set $\cQ'_2:=\{R_1^2,R_2^2\}$.
    If there are no two double rays in any $\cR^1_n$ whose forward tails lie in a common end, then we consider $\cR^1_4$.
    Note that the forward tails of the double rays from $\cR^1_4$ lie in at least four ends $\eta^2_1,\ldots,\eta^2_4$.
    We find a valid sequence $(\cR^{2,1}_n)_{n\in\N}$ induced by $(\cR^1_n)_{n\in\N}$ all of whose forward tails either lie in $\lfloor\eta^2_1\rfloor^*$, or lie in $\lceil\eta^2_1\rceil^*$, or lie in $(\eta^2_1)^\bot$, then we also find such a valid sequence $(\cR^{2,2}_n)_{n\in\N}$ induced by $(\cR^{2,1}_n)_{n\in\N}$ with respect to~$\eta^2_2$, then for~$\eta^2_3$ and finally for~$\eta^2_4$.
    Then there are $i\neq j\in\{1,\ldots,4\}$ such that all forward tails of $(\cR^{2,i}_n)_{n\in\N}$ and $(\cR^{2,j}_n)_{n\in\N}$ either lie in $\lfloor\eta^2_i\rfloor^*$ and $\lfloor\eta^2_j\rfloor^*$, or lie in $\lceil\eta^2_i\rceil^*$ and $\lceil\eta^2_j\rceil^*$, or lie in $(\eta^2_i)^\bot$ and $(\eta^2_j)^\bot$, respectively.
    Let $R^2_1\in\cR^1_4$ whose forward tails lie in~$\eta_i^2$ and let $R_2^2\in\cR^1_4$ whose forward tails lie in~$\eta_j^2$.
    Set $\cQ'_2:=\{R_1^2,R_2^2\}$ and $\cR^2_n:=\cR^{2,4}_n$ for all $n\in\N_{\geq 3}$.
    
    We continue the construction for $3$ instead of~$2$, where the case distinction is whether we find at most $6$ ends that contain all forward tails of all elements of all $\cR^2_n$ or whether there are at least~$7$ such ends.
    So among the valid sequences induced by the $(\cR^{3,i}_n)_{n\in\N}$ in the latter case at least three have the same behaviour with respect to where the ends lie that contain their forward tails.
    We continue this whole construction for all $m\in\N$.
    Then there must be one option that we picked infinitely often: either $\lfloor\eta_j^i\rfloor^*$, or $\lceil\eta_j^i\rceil^*$, or~$(\eta_j^i)^\bot$.
    Building a valid sequence $(\cQ_n)_{n\in\N}$ induced by $(\cQ'_n)_{n\in\N}$ by using the ends where we always did the same choice, this sequence satisfies \ref{itm:ConcentrationSomewhere:2}, if the choice was $\lfloor\eta_j^i\rfloor^*$ infinitely often, or~\ref{itm:ConcentrationSomewhere:3}, if the choice was $\lceil\eta_j^i\rceil^*$ infinitely often, or~\ref{itm:ConcentrationSomewhere:4}, if the choice was $(\eta_j^i)^\bot$ infinitely often.
\end{proof}

This proposition enables us to verify Problem~\ref{prob:ubiquity:doubleRays} for a large class of digraphs.

\begin{thm}\label{thm:ubiquity:doubleRays}
    Let $D$ be a digraph.
    If there is a valid sequence $(\cR_n)_{n\in\N}$ that does not concentrate in finitely many ends, then $D$ contains infinitely many pairwise disjoint double rays.
\end{thm}

\begin{proof}
    Let $(\cR_n)_{n\in\N}$ be a valid sequence in~$D$.
    We distinguish several cases depending on whether there are ends in which the rays or the anti-rays in valid sequences induced by $(\cR_n)_{n\in\N}$ concentrate.
    For this, let us first prove that our assumptions imply that one of the following three cases holds.
    \begin{enumerate}[label = {\rm (\roman*)}]
        \item\label{itm:pf:ubiquity:doubleRays:1} There are distinct ends $\epsilon$ and~$\eta$ and a valid sequence induced by $(\cR_n)_{n\in\N}$ for which the anti-rays concentrate in~$\epsilon$ and the rays concentrate in~$\eta$.
        \item\label{itm:pf:ubiquity:doubleRays:2} There exists an end $\epsilon$ and a valid sequence induced by $(\cR_n)_{n\in\N}$ either whose rays or whose anti-rays concentrate in~$\epsilon$ but such that \ref{itm:pf:ubiquity:doubleRays:1} does not hold.
        \item\label{itm:pf:ubiquity:doubleRays:3} For every end $\epsilon$ and every valid sequence induced by $(\cR_n)_{n\in\N}$, neither the rays nor the anti-rays concentrate in~$\epsilon$.
    \end{enumerate}
    To show that one of \ref{itm:pf:ubiquity:doubleRays:1}--\ref{itm:pf:ubiquity:doubleRays:3} happens, it suffices to show that there is a valid sequence $(\cQ_n)_{n\in\N}$ induced by $(\cR_n)_{n\in\N}$ such that no valid sequence induced by $(\cQ_n)_{n\in\N}$ concentrates in any finite set of ends.
    For all $n\in\N$, let $\cE_n$ be the set of all ends $\epsilon$ such that some double ray $R\in\cR_n$ has all its tails in~$\epsilon$.
    Since $(\cR_n)_{n\in\N}$ does not concentrate in only finitely many ends, there is a subsequence $(\cR_{n_i})_{i\in\N}$ with strictly increasing $n_i$ such that $\cR_{n_i}$ contains at least $i$ double rays each of which either has its tails being equivalent but not in any end from $\bigcup_{j<i}\cE_{n_j}$ or has non-equivalent tails.
    We set $\cQ'_1:=\cR_{n_1}$.
    For $i>1$, let $\cQ'_i$ be the subset of $\cR_{n_i}$ that consists of all double rays with non-equivalent tails, of all double rays whose tails are all equivalent but do not lie in any element of $\bigcup_{j<i}\cE_{n_j}$ and, for all elements $\epsilon$ in $\bigcup_{j<i}\cE_{n_j}$ for which there exists a double ray in $\cR_{n_i}$ with all its tails in~$\epsilon$, exactly one such double ray.
    By construction, $(\cQ'_i)_{i\in\N}$ induces a valid sequence $(\cQ_i)_{i\in\N}$ with $\cQ_i\subseteq \cQ'_i$.
    This sequence clearly satisfies one of \ref{itm:pf:ubiquity:doubleRays:1}--\ref{itm:pf:ubiquity:doubleRays:3}.

    \medskip

    {\bf Case 1.}
    First, let us consider the case that there exist two distinct ends $\epsilon$ and~$\eta$ and a valid sequence $(\cQ_n)_{n\in\N}$ induced by $(\cR_n)_{n\in\N}$ for which the anti-rays concentrate in~$\epsilon$ and the rays concentrate in~$\eta$.
    Then there is no finite vertex set separating $Q$ from~$R$, where $Q$ is an anti-ray in~$\epsilon$ and $R$ is a ray in~$\eta$.
    Thus, there are infinitely many pairwise disjoint $Q$--$R$ dipaths and we have $\epsilon<\eta$.
    Since the anti-rays and rays in $(\cQ_n)_{n\in\N}$ concentrate in~$\epsilon$ and~$\eta$, respectively, there are infinitely many pairwise disjoint anti-rays in~$\epsilon$ and infinitely many pairwise disjoint rays in~$\eta$ by Theorem~\ref{thm:endDegree}.
    Hence, Theorem~\ref{thm:grids} implies the existence of subgraphs $D_1$ and~$D_2$ of~$D$ such that $D_1$ is a subdivision of the reversed bidirected quarter-grid with all anti-rays in~$\epsilon$ and $D_2$ is a subdivision of the bidirected quarter-grid with all rays in~$\eta$.
    Since $\epsilon\neq\eta$, there is a finite vertex set separating either $\eta$ from~$\epsilon$ or $\epsilon$ from~$\eta$ and hence, $D_1$ and $D_2$ have at most finitely many vertices in common.
    So we may assume that they are disjoint.
    Because of $\epsilon<\eta$, there are infinitely many pairwise disjoint $D_1$--$D_2$ dipaths $P_1,P_2,\dots$.
    By Observation~\ref{obs:raysInBiDirQGrid}, there are infinitely many pairwise disjoint anti-rays $R_1,R_2,\ldots$ in~$D_1$ that end at distinct first vertices of those dipaths, say the last vertex of $R_i$ is the first vertex of~$P_{n_i}$.
    Also by Observation~\ref{obs:raysInBiDirQGrid}, we find infinitely many pairwise disjoint rays $Q_1,Q_2,\ldots$ in~$D_2$ such that $Q_i$ starts at the end vertex of the dipath~$P_{n_{k_i}}$.
    Then each $R_{k_i}P_{n_{k_i}}Q_i$ is an $\epsilon$--$\eta$ double ray and every two such double rays are pairwise disjoint.
    Thus, there are infinitely many pairwise disjoint $\epsilon$--$\eta$ double rays.

    \medskip

    {\bf Case 2.}
    Now, we consider the case that there exist an end $\epsilon$ and a valid sequence $(\mathcal Q_n)_{n\in\N}$ induced by $(\mathcal R_n)_{n\in\N}$ either whose rays concentrate in~$\epsilon$ but whose anti-rays do not concentrate in any end or whose anti-rays concentrate in~$\epsilon$ but whose rays do not concentrate in any end.
    We assume that the latter case holds.
    The other case will follow by a completely symmetric argument.
    We may assume $\cQ_n=\cR_n$.
    By Theorem~\ref{thm:grids}, there exists a subdigraph $D'$ of~$D$ such that $D'$ is a subdivision of the reversed bidirected quarter-grid with all anti-rays in~$\epsilon$.
    Furthermore, by Proposition~\ref{prop:ConcentrationSomewhere}, there is a valid sequence $(\cQ_n)_{n\in\N}$ induced by $(\cR_n)_{n\in\N}$ such that for the set $\cE_n$ of ends containing the forward tails of the double rays in~$\cQ_n$ one of \ref{itm:ConcentrationSomewhere:2}, \ref{itm:ConcentrationSomewhere:3}, or \ref{itm:ConcentrationSomewhere:4} from that proposition holds.
    We may assume that $\cQ_n=\cR_n$ for all $n\in\N$.
    We split the remaining proof of this case into the three cases from Proposition~\ref{prop:ConcentrationSomewhere}.
     
    \medskip
    
    So, let us first assume that \ref{itm:ConcentrationSomewhere:4} from Proposition~\ref{prop:ConcentrationSomewhere} holds, that is, that all $\mu\in\cE_i$ and $\eta\in\cE_j$ for $i\neq j$ are incomparable.
    By restricting to a valid sequence induced by $(\cR_n)_{n\in\N}$, if necessary, we may assume that either $\mu>\epsilon$ for all $\mu\in\bigcup_{i\in\N}\cE_i$ or there exists for each $\mu\in\bigcup_{i\in\N}\cE_i$ a finite vertex set $S_\mu$ separating $\mu$ from~$\epsilon$.
    First, we consider the case that $\mu>\epsilon$ for all $\mu\in\bigcup_{i\in\N}\cE_i$.
    Then there exists a ray~$R_1$ starting at~$D'$, that is otherwise disjoint from~$D'$, and that lies in~$\mu_1\in\cE_1$.
    Because of $\mu_2>\epsilon$ for every $\mu_2\in\cE_2$, there is another ray~$R_2$ that starts at~$D'$, that is otherwise disjoint from~$D'$, that lies in some $\mu_2\in\cE_2$, and that does not contain any vertex from~$R_1$: we can pick an arbitrary ray~$R_2'$ in some~$\mu_2$ that is disjoint from~$D'$ and from~$R_1$ and since $\mu_2>\epsilon$ and $\mu_1\not\leq\mu_2$, there must be a finite $D'$--$R_2'$ dipath $P_2$ that avoids~$R_1$ and hence $P_2\cup R_2'$ contains a ray as claimed.
    Recursively, we obtain for every $i\in\N$ a ray $R_i$ that starts at a vertex of~$D'$, that is otherwise disjoint from~$D'$, that lies in some $\mu_i\in\cE_i$, and that is disjoint from all $R_j$ for $j<i$.
    By Observation~\ref{obs:raysInBiDirQGrid}, there exist infinitely many anti-rays $Q_1,Q_2,\ldots$ in~$D'$ whose last vertices are among the starting vertices of an infinite subset of the~$R_i$, say that the last vertex of~$Q_i$ is the first vertex of $R_{n_i}$.
    Then the union of all~$Q_i$ with all $R_{n_i}$ forms an infinite set of pairwise disjoint double rays.
    
    So let us now assume that $\epsilon\not\leq \mu$ for all $\mu\in\bigcup_{i\in\N}\cE_i$, that is, for every $\mu\in\bigcup_{i\in\N}\cE_i$, there is a finite vertex set $S_\mu$ that separates $\mu$ from~$\epsilon$.
    Let $\{R_1^1\}:=\cR_1$ and let $\mu_1$ be the unique element of~$\cE_1$.
    Since there are infinitely many pairwise disjoint $D'$--$R^1_1$ dipaths, let $P^1$ be one of them that avoids $S_1:=S_{\mu_1}$ and such that no vertex of~$S_1$ lies on~$R_1^1$ before the end vertex of~$P^1$.
    Then there exists a ray $R_1$ in $P^1\cup R^1_1$ that starts at the first vertex of~$P^1$ and, since $S_1$ separates $\mu_1$ from~$\epsilon$, there is a subdipath $P_1$ of~$R_1$ with the same first vertex as~$R_1$ whose last vertex is the last vertex on $R_1^1$ in~$S_1$.
    Set $i_2:= |P_1|+|S_1|+1$ and $\{R_2^1,\ldots,R_2^{i_2}\}:=\cR_{i_2}$.
    Then there is some $R_2^j$ that avoids $V(P_1)\cup S_1$.
    As $S_1$ separates $\mu_1$ from~$\epsilon$, the double ray $R_2^j$ is disjoint from~$R_1-P_1$ and hence from the whole ray~$R_1$.
    As in the previous case, we find a $D'$--$R_2^j$ dipath $P^2$ that meets $R_2^j$ before any vertices from~$S_2:=S_{\mu_2}$, where $\mu_2$ is the end that contains the forward tail of~$R_2^j$.
    Let $R_2$ be a ray in $P^2\cup R_2^j$ with the same first vertex as~$P^2$ and let $P_2$ be the subdipath of~$R_2$ that has the same first vertex as~$R_2$ and ends at the last vertex on~$R_2^j$ in~$S_2$.
    We set $i_3:=|P_1|+|S_1|+|P_2|+|S_2|+1$ and continue this construction recursively.
    Thus, we obtain infinitely many disjoint rays $R_i$ that start at (distinct) vertices of~$D'$, that are otherwise disjoint from~$D'$, and that lie in elements of~$\cE_i$.
    By Observation~\ref{obs:raysInBiDirQGrid}, there are infinitely many pairwise disjoint anti-rays in~$D'$ that we can join to a suitable infinite subset of $\{R_i\mid i\in\N\}$ to obtain infinitely many pairwise disjoint double rays.
    This completes the subcase that \ref{itm:ConcentrationSomewhere:4} from Proposition~\ref{prop:ConcentrationSomewhere} holds.

    \medskip

    Now, let us assume that \ref{itm:ConcentrationSomewhere:2} from Proposition~\ref{prop:ConcentrationSomewhere} holds, that is, we have $\mu<\eta$ for all $\mu\in\cE_i$ and all $\eta\in\cE_j$ if $i<j$.
    In particular, we find a countably infinite ascending sequence $\mu_1<\mu_2<\ldots$ with $\mu_i\in\cE_i$.
    By Theorem~\ref{thm:Zuther:grids}, there exists an end~$\mu$ with infinitely many pairwise disjoint rays that is the supremum of the~$\mu_i$.
    Note that $\eta<\mu_{i+1}<\mu$ holds for every $\eta\in\cE_i$.
    Let us first assume that $\mu\neq\epsilon$.
    Let $D''$ be a subdivision of the bidirected quarter-grid with all its rays in~$\mu$.
    By an argument used above, we may assume that $\mu$ and $\epsilon$ are disjoint since they are distinct.
    Let $R_1\in\cR_1$.
    Then there are $D'$--$R_1$ and $R_1$--$D''$ dipaths $P'_1$ and $P''_1$, respectively, such that $P'_1\cup R_1\cup P''_1$ contains a $D'$--$D''$ dipath~$P_1$.
    Note that the $R_1$--$D''$ dipath exists as $\mu_1<\mu$.
    Set $i_2:=|P_1|+1$.
    Then there exists $R_2\in\cR_{i_2}$ that avoids~$P_1$.
    Since the backward tail of~$R_2$ lies in~$\epsilon$, there exists a $D'$--$R_2$ dipath that avoids~$P_1$, too.
    We already saw $\eta_2<\mu$ for the end $\eta_2$ that contains the forward tail of~$R_2$.
    Thus, there exists an $R_2$--$D''$ dipath $P''_2$ that avoids~$P_1$ and such that $P'_2\cup R_2\cup P''_2$ contains a $D'$--$D''$ dipath~$P_2$.
    By construction, $P_1$ and~$P_2$ are disjoint.
    We continue with $i_3:=|P_1|+|P_2|+1$ and recursively obtain infinitely many pairwise disjoint $D'$--$D''$ dipaths~$P_i$.
    By Observation~\ref{obs:raysInBiDirQGrid}, there is an infinite subset of those paths that can be extended to rays with tails in~$D''$.
    By the same lemma, we can extend infinitely many of these rays by anti-rays in~$D'$ to double rays.
    Thus, $D$ contains infinitely many pairwise disjoint double rays.
    
    So, let us now assume that $\mu=\epsilon$.
    Then there exists some ray~$R_1$ that starts at~$D'$ and lies in~$\eta_1\in\cE_1$.
    Since $\eta_1<\mu$, there exists a finite vertex set~$S_1$ that separates $\eta_1$ from~$\mu$.
    If $R_1$ contains a vertex of~$S_1$, then let $P_1$ be the initial subdipath of~$R_1$ whose last vertex is the last common vertex of~$R_1$ and~$S_1$.
    If $R_1$ contains no vertex of~$S_1$, then let $P_1$ be the trivial dipath consisting of the first vertex of~$R_1$.
    Set $i_2:=|P_1|+|S_1|+1$.
    Then there exists a double ray $Q_2\in\cR_{i_2}$ that avoids $S_1$ and~$P_1$ and that ends in~$\eta_2\in\cE_{i_2}$.
    Some $D'$--$Q_2$ dipath avoids $S_1$ and~$P_1$, too.
    Furthermore, as $\eta_1<\epsilon$, we may assume that this $D'$--$Q_2$ dipath also avoids~$R_1$.
    So there exists a ray $R_2$ in~$\eta_2$ that has exactly its first vertex in~$D'$ and that avoids $S_1$ and~$P_1$ and is disjoint from~$R_1$.
    As in the previous case, we recursively construct infinitely many pairwise disjoint rays that start at distinct vertices of~$D'$ and end in distinct ends~$\eta_i$.
    By Observation~\ref{obs:raysInBiDirQGrid}, there are infinitely many pairwise disjoint anti-rays in~$D'$ whose end vertices are among the first vertices of the rays~$R_i$.
    Thus, we find infinitely many pairwise disjoint double rays by joining those anti-rays with a suitable infinite subset of $\{R_i\mid i\in\N\}$.
    This completes the subcase that \ref{itm:ConcentrationSomewhere:2} from Proposition~\ref{prop:ConcentrationSomewhere} holds.

    \medskip

    As the last subcase of Case 2, we now assume that \ref{itm:ConcentrationSomewhere:3} from Proposition~\ref{prop:ConcentrationSomewhere} holds, that is, we have $\eta<\mu$ for all $\mu\in\cE_i$ and all $\eta\in\cE_j$ if $i<j$.
    By passing to a valid sequence induced by $(\cR_n)_{n\in\N}$, we may assume that the ends in $\bigcup_{i\in\N}\cE_i$ either all lie above~$\epsilon$, or all lie below~$\epsilon$ or all are incomparable with~$\epsilon$.
    Let $\mu_1\in\cE_1$.
    If $\epsilon\not\leq\mu_1$, then there exists a finite vertex set $S$ such that every $\epsilon$--$\mu_1$ double ray meets~$S$.
    In $\mathcal R_s$ for $s:=|S|+1$ at least one double ray $R$ avoids~$S$.
    Let $\eta$ be the end that contains a forward tail of~$R$.
    Let $R'$ be a ray in~$\mu_1$ that avoids~$S$.
    Because of $\eta\leq\mu_1$, there are infinitely many pairwise disjoint $R$--$R'$ dipaths, one of which avoids~$S$.
    Thus, there exists an $\epsilon$--$\mu_1$ double ray avoiding~$S$, which is a contradiction that shows $\epsilon\leq\mu_1$.
    Analogously, we have $\epsilon\leq\eta$ for all $\eta\in\bigcup_{i\in\N}\cE_i$.
    For every $i\in\N$, let $\mu_i\in\cE_i$ and let $R^i$ be a ray in~$\mu_i$ such that no~$R^i$ contains vertices from~$D'$ and such that all these rays are pairwise disjoint.
    We recursively choose $D'$--$R^i$ dipaths $P_i$ and find rays $R_i$ in~$\mu_i$ that start at a vertex of~$D'$ and are otherwise disjoint from~$D'$ and such that all rays $R_i$ are pairwise disjoint.
    For this, we start with an arbitrary $D'$--$R^1$ dipath $P_1$ and choose a ray in $P_1\cup R^1$ that starts at~$D'$.
    For general $i$, there exists a $D'$--$R^i$ dipath $P_i$ such that the ray $R_i$ in $P_i\cup R^i$ that starts at the same vertex as~$P_i$ avoids all $R_j$ for $j<i$ because of $\epsilon<\mu_i<\mu_j$.
    These rays clearly satisfy the properties as claimed.
    By taking an infinite subset of them for which we find by Observation~\ref{obs:raysInBiDirQGrid} infinitely many pairwise disjoint anti-rays in~$D'$ that end at the starting vertices of those rays, we obtain infinitely many pairwise disjoint double rays in~$D$.
    This finishes the case that \ref{itm:ConcentrationSomewhere:3} from Proposition~\ref{prop:ConcentrationSomewhere} holds and it finishes Case~2.

    \medskip

    {\bf Case 3.}
    Let us now consider the remaining case:
    For every end $\epsilon$ and every valid sequence $(\mathcal Q_{n_i})_{i\in\N}$ induced by $(\cR_n)_{n\in\N}$, neither the rays nor the anti-rays concentrate in~$\epsilon$.
    Then we may assume that all double rays from distinct sets~$\mathcal R_i$ have their tails which are anti-rays in distinct ends and have their tails which are rays in distinct ends.
    So by Proposition~\ref{prop:ConcentrationSomewhere}, there is a valid sequence $(\cQ_n)_{n\in\N}$ induced by $(\cR_n)_{n\in\N}$ such that for the set $\cE_n$ of ends containing the forward tails of the double rays in~$\cQ_n$ one of \ref{itm:ConcentrationSomewhere:2}, \ref{itm:ConcentrationSomewhere:3}, or \ref{itm:ConcentrationSomewhere:4} from that proposition holds.
    We may assume that $\cQ_n=\cR_n$ for all $n\in\N$.
    We split the remaining proof of this case into which of those three cases from Proposition~\ref{prop:ConcentrationSomewhere} holds.
    Analogously, we may assume that for the backward tails of the double rays in each~$\cR_n$ and for the set $\cE'_n$ of ends containing these tails, one of \ref{itm:ConcentrationSomewhere:2}, \ref{itm:ConcentrationSomewhere:3}, or \ref{itm:ConcentrationSomewhere:4} from Proposition~\ref{prop:ConcentrationSomewhere} holds.

    \medskip
    
    We start with the case \ref{itm:ConcentrationSomewhere:2} from Proposition~\ref{prop:ConcentrationSomewhere}, that is, we have $\mu<\eta$ for all $\mu\in\cE_i$ and all $\eta\in\cE_j$ if $i<j$.
    For all $i\in\N$, let $\mu_i\in\cE_i$.
    So we have $\mu_1<\mu_2<\dots$.
    By Theorem~\ref{thm:Zuther:grids}, there is a supremum $\mu$ of the chain $(\mu_i)_{i\in\N}$ containing infinitely many pairwise disjoint rays.
    If, for each $\mu_i$, there are only finitely many ends $\eta$ that contain anti-rays and satisfy $\mu_i<\eta<\mu$, then we can assume, by passing to a valid sequence induced by $(\cR_i)_{i\in\N}$, that there are no such ends $\eta$ at all.
    Let $T_1,\ldots,T_i$ be $i$ pairwise disjoint rays in~$\mu$.
    Since $\epsilon<\mu_{i+1}\leq\mu$ for all $\epsilon\in\cE_i$, there exists a set $\cP$ consisting of $i$ dipaths from each element of~$\cR_i$ to each $T_j$ that are all pairwise disjoint.
    By Menger's theorem there exist a set $\cQ_i$ consisting of $i$ pairwise disjoint double rays in the union of the elements of~$\cR_i$, of the $T_j$ and of the elements of~$\cP$ such that their forward tails lie in~$\mu$ and the backward tails coincide with backward tails from elements of~$\cR_i$.
    Thus, there is a valid sequence induced by $(\cQ_i)_{i\in\N}$ that concentrates in~$\mu$.
    So, the assumption of Case 2 is satisfied and, hence, we find infinitely many pairwise disjoint double rays in~$D$.
    
    So, we may  assume that for every $\mu_i$ there exists an end $\eta_i$ with $\mu_i<\eta_i<\mu$ that contains anti-rays.
    In particular, there are infinitely many pairwise disjoint anti-rays $R_1,R_2,\ldots$ that lie in distinct ends $\epsilon<\mu$.
    By Theorem~\ref{thm:grids}, there exists a subdigraph $D'$ of~$D$ such that $D'$ is a subdivision of the bidirected quarter-grid with all rays in~$\mu$.
    Since each $R_i$ lies in an end $\epsilon<\mu$, we recursively find dipaths $P_i$ from $R_i$ to~$D'$ and tails $R_i'$ of~$R_i$ such that $P_i \cup R'_i$ and $P_j \cup R'_j$ are disjoint for all $i \neq j$.
    By Observation~\ref{obs:raysInBiDirQGrid}, there is an infinite index set $(j_i)_{i\in\N}$ such that we can extend an anti-ray in $R_{j_i}\cup P_{j_i}$ by a suitable ray in~$D'$ to obtain an infinite set of pairwise disjoint double rays.

    The case that \ref{itm:ConcentrationSomewhere:3} from Proposition~\ref{prop:ConcentrationSomewhere} holds for the backward tails of the~$\cR_i$ is analogous, that is, if $\mu>\eta$ for all $\mu\in\cE'_i$ and all $\eta\in\cE'_j$ with $i<j$, then we find infinitely many pairwise disjoint double rays in~$D$.
    
    \medskip

    Let us now assume that \ref{itm:ConcentrationSomewhere:3} from Proposition~\ref{prop:ConcentrationSomewhere} holds, that is, we have $\mu>\eta$ for all $\mu\in\cE_i$ and all $\eta\in\cE_j$ if $i<j$.
    By the previous situation, we may assume that \ref{itm:ConcentrationSomewhere:3} from Proposition~\ref{prop:ConcentrationSomewhere} does not hold for the backward tails of the~$\cR_i$.
    Thus, we have $\eta_j\not\leq \eta_i$ for all $\eta_j\in\cE'_j$, all $\eta_i\in\cE'_i$ and $i<j$.
    If there are, for every $j\in\N$, infinitely many $i\in\N$ such that $\cE'_i$ contains an element $\eta<\mu_j\in\cE_j$, then we pick an arbitrary $\eta_1$--$\mu_1$ double ray~$R_1$, where $\eta_1$ lies in~$\cE'_{n_1}$ and $\mu_1\in\cE_1$.
    Now we pick some $\mu_2\in\cE_2$ and some $\eta_2\in\cE'_{n_2}$ with $n_2>n_1$ such that $\eta_2< \mu_2<\mu_1$ and $\eta_2\not\leq\eta_1$ hold.
    We want to construct an $\eta_2$--$\mu_2$ double ray $R_2$ disjoint from~$R_1$.
    If for some anti-ray in~$\eta_2$ and ray in~$\mu_2$ there are infinitely many pairwise disjoint dipaths from the anti-ray to the ray all of which meet~$R_1$, then they must meet some forward tail of~$R_1$ since $\eta_2\not\leq\eta_1$.
    So $\eta_2\leq\mu_1\leq\mu_2$ must hold, which contradicts our assumptions.
    Thus, all but finitely many of the dipaths are disjoint from~$R_1$.
    So we can find an $\eta_2$--$\mu_2$ double ray $R_2$ disjoint from~$R_1$.
    Similarly, we find an $\eta_j$--$\mu_j$ double ray~$R_j$ that is disjoint from all $R_k$ with $k<j$.
    Thus, we find infinitely many pairwise disjoint double rays in~$D$.

    So, we may assume that there are, for each $j\in\N$, only finitely many $i\in\N$ such that $\eta<\mu$ for some $\eta\in\cE'_j$ and $\mu\in\cE_i$.
    As \ref{itm:ConcentrationSomewhere:3} from Proposition~\ref{prop:ConcentrationSomewhere} holds for the forward tails in the~$\cR_i$, we may assume by switching to a suitable valid sequence induced by $(\cR_n)_{n\in\N}$ that there are no $i,j\in\N$ such that $\eta<\mu$ for any $\eta\in\cE'_j$ and $\mu\in\cE_i$.
    Let $R_1\in\cR_1$ and let $S_1$ be a finite vertex set that separates $\cE_2$ from~$\cE'_1$ and~$\cE_1$, which exists by the choice of our valid sequence.
    Note that then $S_1$ also separates every $\cE_i$ for $i\geq 2$ from~$\cE'_1$ and~$\cE_1$.    
    Let us suppose that there exists no vertex $v$ on~$R_1$ such that $R_1v$ avoids all double rays in $\bigcup_{i>1}\cR_i$ that lie in $D-S_1$.
    Since the end that contains the forward tail of every $R'\in \bigcup_{i>1}\cR_i$ that meets $R_1$ but avoids~$S_1$ lies below the end in~$\cE_1$, we find a ray $R''$ equivalent to the forward tail of~$R_1$ and an $R'$--$R''$ dipath $P_1$ whose first vertex lies after some vertex in $R_1\cap R'$ such that $R'$ and $P_1$ avoid~$S_1$.
    Then the union of $R_1$, $R'$, $R''$ and $P_1$ contains a double ray from an end in~$\cE'_1$ to an end in~$\cE_2$ that avoids~$S_1$, a contradiction.
    Thus, there exists $v$ on~$R_1$ such that $R_1v$ avoids all double rays in $\bigcup_{i>1}\cR_i$ that lie in $D-S_1$ and, if $R_1$ contains vertices from~$S_1$, then we may assume that $v$ lies before the first vertex of~$S_1$ on~$R_1$.
    If there exists no vertex $w$ on~$R_1$ and no finite vertex set $S'_1$ such that $wR_1$ avoids all double rays in $\bigcup_{i>1}\cR_i$ that lie in $D-(S_1\cup S'_1)$, then we can use these double rays and the fact that their forward tails lie in ends in $\bigcup_{i > 1} \cE_i$ to find infinitely many disjoint dipaths from a forward tail of~$R_1$ to some ray in $\epsilon\in\cE_2$.
    This implies that $\epsilon$ is larger than the end that contains the forward tail of~$R_1$ and, hence, these two ends must be the same, which contradicts that \ref{itm:ConcentrationSomewhere:3} from Proposition~\ref{prop:ConcentrationSomewhere} holds.
    Thus, there exists $w$ on~$vR_1$ and a finite vertex set $S'_1$ such that $wR_1$ avoids all double rays in $\bigcup_{i>1}\cR_i$ that lie in $D-(S_1\cup S'_1)$.
    Let $i_2>|V(vR_1w)|+|S_1|+|S'_1|$.
    One of the double rays $R_2\in\cR_{i_2}$ must be disjoint from~$vR_1w$ and from $S_1\cup S'_1$.
    Note that $R_2$ is also disjoint from $R_1v$ and $wR_1$ by the choices of~$v$ and~$w$, respectively, so it is disjoint from~$R_1$.
    We continue by taking a finite vertex set $S_2$ that separates $\cR_{i_2+1}$ from the four ends that contain the tails of~$R_1$ and~$R_2$, finding vertices on~$R_1$ and~$R_2$ similar to~$v$ and~$w$ and a vertex set $S'_2$ similar to~$S'_1$ and then finding a suitable double ray~$R_3$.
    Recursively, we find infinitely many pairwise disjoint double rays in~$D$.

    The case that \ref{itm:ConcentrationSomewhere:2} from Proposition~\ref{prop:ConcentrationSomewhere} holds for the backward tails of the~$\cR_i$ is analogous, that is, if $\mu<\eta$ for all $\mu\in\cE'_i$ and all $\eta\in\cE'_j$ if $i<j$, then we find infinitely many pairwise disjoint double rays in~$D$.

    \medskip

    It remains to consider the situation \ref{itm:ConcentrationSomewhere:4} from Proposition~\ref{prop:ConcentrationSomewhere} holds for the $\cE_i$'s as well as the $\cE'_i$'s, that is, $\eta$ and~$\mu$ are incomparable for all $\eta\in\cE_i$ and $\mu\in\cE_j$ for $i\neq j$ and for all $\eta\in\cE'_i$ and $\mu\in\cE'_j$ for $i\neq j$.
    Let $R_1\in\cR_1=:\cP_1$.
    Then there is a valid sequence $(\cR^1_n)_{n\in\N}$ induced by $(\cR_n)_{n\in\N}$ such that either all $R\in\bigcup_{n>1}\cR_n^1$ intersect $R_1$ or all of them are disjoint from~$R_1$.
    We consider $i_2=3$ and let $\{R_1^2,R_2^2,R_3^2\}=\cR^1_{i_2}$.
    There is again a valid sequence $(\cR^{2,1}_n)_{n\in\N}$ induced by $(\cR^1_n)_{n\in\N}$ all of whose double rays behave the same to~$R^2_1$ with respect to intersection, the same for a valid sequence $(\cR^{2,2}_n)_{n\in\N}$ induced by $(\cR^{2,1}_n)_{n\in\N}$ for $R^2_2$ and the same for a valid sequence $(\cR^{2,3}_n)_{n\in\N}$ induced by $(\cR^{2,2}_n)_{n\in\N}$ for~$R^3_2$.
    For $i\neq j$ two of these sequences show the same behaviour with respect to~$R^2_i$ and~$R^2_j$.
    We set $\cP_2:=\{R^2_i,R^2_j\}$ and continue the construction for $i_3=5$ and so on.
    In the end, one type of behaviour (either trivial or non-trivial intersection) was used infinitely often to build the induced sequences and thus, there is a valid sequence $(\cQ_n)_{n\in\N}$ induced by $(\cP_n)_{n\in\N}$ where either all elements of $\bigcup_{n\in\N}\cQ_n$ pairwise intersect or they are pairwise disjoint.
    In the second case, we directly have infinitely many pairwise disjoint double rays.
    So let us assume that all elements of $\bigcup_{n\in\N} \cQ_n$ pairwise intersect.
    In this case, we fix double rays $Q_i\in\cQ_i$ for $i\in\{1,2,3\}$ with $Q_i=A_i\cup B_i$ for an anti-ray $A_i$ and a ray~$B_i$.
    Since the $4j$ double rays in~$\cQ_{4j}$ are pairwise disjoint, there exists, for each $i\in\{1,2,3\}$, $2j$ double rays $Q^k_i\in\cQ_{4j}$ such that $Q^k_i$ contains no vertices among the last $j$ or first $j$ vertices of~$A_i$ or~$B_i$, respectively, but contains a vertex either on~$A_i$ of distance more than $j$ to the last vertex of~$A_i$ or on~$B_i$ of distance more than $j$ from the first vertex of~$B_i$.
    In particular, either $j$ such double rays contain vertices on~$A_i$ or $j$ such double rays contain vertices on~$B_i$.
    Thus, there exist distinct $k,\ell\in\{1,2,3\}$ such that for every $j\in\N$ there are either $j$ pairwise disjoint $A_k$--$A_\ell$ or $B_k$--$B_\ell$ dipaths that originate from double rays of~$\cQ_{4j}$.
    This implies either $A_k\leq A_\ell$ or $B_k\leq B_\ell$, which is impossible, as the ends that contain these anti-rays or rays are incomparable since \ref{itm:ConcentrationSomewhere:4} from Proposition~\ref{prop:ConcentrationSomewhere} holds for the $\cE_i$'s as well as the $\cE'_i$'s.
    This contradiction finishes the third case and thus the proof of the theorem.
\end{proof}

Let us now prove a lemma similar to Lemma~\ref{lem:reduction:oneEnd} but for double rays instead of rays and anti-rays.

\begin{lemma}\label{lem:reduction:doubleRays:oneEnd}
    Let $D$ be a digraph and let $(\cR_n)_{n\in\N}$ be a valid sequence.
    If there is an end~$\epsilon$ such that all tails of all $R\in\bigcup_{n\in\N}\cR_n$ lie in~$\epsilon$, then there is a one-ended subdigraph $D'$ of~$D$ such that there is a valid sequence $(\cQ_n)_{n\in\N}$ induced by $(\cR_n)_{n\in\N}$ in~$D'$.
\end{lemma}

\begin{proof}
    For all $n\in\N$ and all $1\leq i\leq n$, let $R_i^n$ be the elements of $\cR_n$ and let $R_i^n=A_i^n\cup B_i^n$, where $A_i^n$ is an anti-ray, $B_i^n$ is a rays and they share a unique vertex~$v_i^n$.
    Let $\cQ_1=\cR_1$ and set $S_1:=\{v_1^1\}$.
    Then there are two elements in~$\cR_3$ that are disjoint from $S_1$ and we let $\cQ_2$ be the set of these two double rays.
    We let $S_2$ be the set consisting of $v_1^1$ together with its in- and out-neighbour in~$R_1^1$ and of $v_i^3$ and $v_j^3$.
    In $\cR_{|S_2|+3}$, we find three double rays that avoid~$S_2$.
    Let $\cQ_3$ be the set of those double rays.
    We construct $S_3$ analogously such that we enlarge the in- and out-balls in~$R_i^n$ around the vertices $v_i^n$ by~$1$ in each step.
    Thereby, we construct a valid sequence $(\cQ_n)_{n\in\N}$ induced by $(\cR_n)_{n\in\N}$ such that every vertex lies in only finitely many double rays from $\bigcup_{n\in\N}\cQ_n$.
    Now we enumerate all ordered pairs of elements from $\bigcup_{n\in\N}\cQ_n$ such that every pair occurs infinitely often.
    We then recursively add dipaths from the rays $B_i^n$ to the anti-rays $A_j^m$ and from the anti-rays $A_j^m$ to the rays $B_i^n$ from each pair such that all these dipaths are pairwise disjoint.
    This is possible as all rays and anti-rays involved are equivalent in~$D$.
    We denote the resulting subdigraph of~$D$ by~$D'$.
    By construction, it only remains to show that $D'$ is one-ended.
    This, however, follows from the choice of the dipaths: as every pair of double rays occurred infinitely often in the enumeration, all $A_i^n$ and $B_j^m$ are equivalent in~$D'$.
    Furthermore, as the additional dipaths are all pairwise disjoint, every ray or anti-ray $R$ in~$D'$ must either meet some $A_i^n$ or some $B_i^m$ infinitely often or must meet infinitely many $A_i^n$ or~$B_j^m$.
    While in the first case $R$ is equivalent to~$A_i^n$ or~$B_j^m$ directly, we can use the infinitely many double rays that $R$ meets to find infinitely many pairwise disjoint dipaths to and from~$A_1^1$ in the other case.
    This is possible since each vertex lies in only finitely many double rays from $\bigcup_{n\in\N}\cQ_n$.
    Thus, $D'$ is one-ended.
\end{proof}

Now we can prove Theorem~\ref{thm:main:intro:ReductionToOneEnd:doubleRays}.

\begin{proof}[Proof of Theorem~\ref{thm:main:intro:ReductionToOneEnd:doubleRays}.]
    Let us assume that Problem~\ref{prob:ubiquity:doubleRays} holds for the class of one-ended digraphs and let $D$ be a digraphs and $(\cR_n)_{n\in\N}$ be a valid sequence in~~$D$.
    By Theorem~\ref{thm:ubiquity:doubleRays}, we may assume that $(\cR_n)_{n\in\N}$ concentrates in finitely many ends.
    Then there is a valid sequence $(\cQ_n)_{n\in\N}$ induced by $(\cR_n)_{n\in\N}$ such that all tails from all elements of $\bigcup_{n\in\N}\cQ_n$ are equivalent.
    By Lemma~\ref{lem:reduction:doubleRays:oneEnd}, there exists a one-ended subdigraph $D'$ of~$D$ and a valid sequence $(\cP_n)_{n\in\N}$ in~$D'$ induced by $(\cQ_n)_{n\in\N}$.
    Thus, there are infinitely many pairwise disjoint double rays in~$D'$ and hence in~$D$.
\end{proof}

\section{On the equivalence of ubiquity problems}\label{sec:equivalence}

In this section, we discuss connections between Problems~\ref{prob:ubiquity:doubleRays} and~\ref{prob:ubiquity:antiraysRays}.
We also introduce two more problems related to them.

\begin{thm}\label{thm:equivalence:probsFromIntro}
    Every counterexample to Problem~\ref{prob:ubiquity:antiraysRays} is also a counterexample to Problem~\ref{prob:ubiquity:doubleRays}.

    Furthermore, if every one-ended digraph that contains infinitely many rays and infinitely many anti-rays that are all pairwise disjoint also contains infinitely many pairwise disjoint double rays, then there exists a counterexample to Problem~\ref{prob:ubiquity:antiraysRays} if and only if there exists one to Problem~\ref{prob:ubiquity:doubleRays}.
\end{thm}

\begin{proof}
    Let $D$ be a digraph.
    Let us first assume that $D$ is not a counterexample to Problem~\ref{prob:ubiquity:doubleRays}.
    For all $n\in\N$, let $\cX_n$ be a set consisting of~$n$ rays and $n$ anti-rays in~$D$ all of which are pairwise disjoint.
    By Theorem~\ref{thm:ubiquity:antiraysRays}, we may assume that all elements from all sets~$\cX_n$ lie in a common end~$\epsilon$.
    By Lemma~\ref{prop:rays+anti=double}, there exists for every $n\in\N$ a set $\cY_n$ of $n$ disjoint double rays all of whose tails lie in~$\epsilon$.
    Since $D$ is not a counterexample to Problem~\ref{prob:ubiquity:doubleRays}, there exist infinitely many pairwise disjoint double rays in~$D$.
    By taking two disjoint tails from each double ray, we obtain a set consisting of infinitely many rays and infinitely many anti-rays such that all its elements are pairwise disjoint, which shows that $D$ is not a counterexample to Problem~\ref{prob:ubiquity:antiraysRays}.

    By the previous argument, it suffices for the additional statement to show the backward implication.
    Let us assume that Problem~\ref{prob:ubiquity:doubleRays} has a counterexample $D$, but Problem~\ref{prob:ubiquity:antiraysRays} has none.
    In particular, $D$ is no counterexample to Problem~\ref{prob:ubiquity:antiraysRays}.
    Let $(\cR_n)_{n\in\N}$ be a valid sequence in~$D$.
    By Theorem~\ref{thm:ubiquity:doubleRays}, we may assume that $(\cR_n)_{n\in\N}$ concentrates in finitely many ends.
    Thus, there exists a valid sequence $(\cQ_n)_{n\in\N}$ such that all tails of elements from the sets~$\cQ_n$ lie in a common end~$\epsilon$.
    In particular, we find sets $\cX_n$, for all $n\in\N$, consisting of $n$ rays and $n$ anti-rays all of which are pairwise disjoint and all of which lie in~$\epsilon$.
    By Lemma~\ref{lem:reduction:oneEnd}, there exists a one-ended subdigraph $D'$ of~$D$ and, for every $n\in\N$, a set $\cY_n$ consisting of $n$ rays and $n$ anti-rays in~$D'$ that are all pairwise disjoint.
    By assumption, $D'$ is no counterexample to Problem~\ref{prob:ubiquity:antiraysRays}.
    Hence, $D'$ contains a set of infinitely many rays and infinitely many anti-rays that are all pairwise disjoint.
    Then the assumption implies the existence of infinitely many pairwise disjoint double rays in~$D'$ and hence in~$D$, which shows that $D$ is not a counterexample to Problem~\ref{prob:ubiquity:doubleRays} which contradicts the choice of~$D$.
\end{proof}

Let us now state two more problems.

\begin{problem}\label{prob:ubiquity:1/2:antiraysRays}
    Let $m\in\N$.
    If a digraph contains for all $n\in\N$ a set of $m$ rays and $n$ anti-rays that are pairwise disjoint, does it contain a set of $m$ rays and infinitely many anti-rays that are pairwise disjoint?
\end{problem}

\begin{problem}\label{prob:ubiquity:2/2:antiraysRays}
    If a digraph contains for all $n\in\N$ a set of $n$ rays and infinitely many anti-rays that are pairwise disjoint, does it contain a set of infinitely many rays and infinitely many anti-rays that are pairwise disjoint?
\end{problem}

These two problems are strongly related to Problem~\ref{prob:ubiquity:antiraysRays} as we shall discuss now.

\begin{thm}
    Let $D$ be a digraph that contains for each $n\in\N$ a set consisting of $n$ rays and $n$ anti-rays in~$D$ that are all pairwise disjoint.
    Then $D$ is a counterexample to Problem~\ref{prob:ubiquity:antiraysRays} if and only if it is a counterexample either to Problem~\ref{prob:ubiquity:1/2:antiraysRays} for some $m\in\N$ or to Problem~\ref{prob:ubiquity:2/2:antiraysRays}.
\end{thm}

\begin{proof}
    If $D$ is a not counterexample to Problem~\ref{prob:ubiquity:1/2:antiraysRays}, for any $m\in\N$, and not to Problem~\ref{prob:ubiquity:2/2:antiraysRays}, then the existence, for each $n\in\N$, of a set of $n$ rays and $n$ anti-rays that are all pairwise disjoint implies that the assumptions of Problem~\ref{prob:ubiquity:2/2:antiraysRays} are true and hence there is a set with infinitely many rays and infinitely many anti-rays that are all pairwise disjoint.
    Thus, $D$ is not a counterexample to Problem~\ref{prob:ubiquity:antiraysRays}.

    Lastly, if $D$ is not a counterexample to Problem~\ref{prob:ubiquity:antiraysRays}, then this directly implies that it is for no $m\in\N$ a counterexample to Problem~\ref{prob:ubiquity:1/2:antiraysRays}.
    So, the assumption and conclusion of Problem~\ref{prob:ubiquity:2/2:antiraysRays} and thereby $D$ is not a counterexample to that problem, either.
\end{proof}

\section*{Acknowledgement}

We thank Nathan Bowler for shortening an argument at the beginning of the proof of Theorem 4.2.

\bibliographystyle{amsplain}
\bibliography{Ubiquity}

\end{document}